\documentclass[11pt]{amsart}
\usepackage[latin1]{inputenc}
\usepackage{graphicx}
\usepackage{amsmath}
\usepackage{amssymb}

\usepackage{amsmath,amsthm,amsfonts,amssymb,amscd}
\usepackage{enumerate}% http://ctan.org/pkg/enumerate
\usepackage{amssymb,amsthm,amsmath,psfrag,mathrsfs}
\usepackage[latin1]{inputenc}
\usepackage{graphicx}
\usepackage{comment}

\def\leaderfill{\leaders\hbox to .8em{\hss .\hss}\hfill}
\def\_#1{{\lower 0.7ex\hbox{}}_{#1}}

\newtheorem{theorem}{Theorem}

\newtheorem{Theorem}{Theorem}[section]

\newtheorem{Proposition}[Theorem]{Proposition}
\newtheorem{Lemma}[Theorem]{Lemma}
\newtheorem{Claim}[Theorem]{Claim}
\theoremstyle{Definition}
\newtheorem{Definition}[Theorem]{Definition}

\theoremstyle{Remark}
\newtheorem{Remark}[Theorem]{Remark}

\def\leaderfill{\leaders\hbox to .8em{\hss .\hss}\hfill}
\def\_#1{{\lower 0.7ex\hbox{}}_{#1}}

\def\fa{{\mathcal{F}}}

\def\ov{\overline}

\def\bc{{\mathbb{C}}}

\def\Re{\operatorname{{Re}}}

\def\mI{\operatorname{{Im}}}

\def\hot{\operatorname{h.o.t.}}

\def\hot{\operatorname{h.o.t.}}

\def\Diff{\operatorname{{Diff}}}
\def\sing{\operatorname{{sing}}}

\title[On real  center singularities of complex vector fields on surfaces]{On real  center singularities of complex vector fields on surfaces}
\author{V. Le\'on and B. Sc\'ardua}
\address{V. Le\'on. ILACVN - CICN, Universidade Federal da Integração Latino-Americana, Parque tecnológico de Itaipu, Foz do Iguaçu-PR, 85867-970 - Brazil}
\email{victor.leon@unila.edu.br}
\address{B. Sc\'ardua. Instituto de Matem\'atica - Universidade Federal do Rio de Janeiro,
CP. 68530-Rio de Janeiro-RJ, 21945-970 - Brazil}
\email{bruno.scardua@gmail.com}

\keywords{foliation; center singularity; first integral; integrable form; Reeb theorem.}
\date{}
\begin{document}

\maketitle

\begin{abstract}One of the various versions of the classical Lyapunov-Poincaré center theorem states that a nondegenerate real analytic center type planar vector field singularity admits an analytic first integral. In a more proof of this result, R. Moussu establishes  important  connection between this result and the theory of singularities of holomorphic foliations (\cite{moussu}).
In this paper we consider generalizations for two main frameworks:
(i)  planar  real analytic vector fields with ``many'' periodic orbits near the singularity  and (ii) germs of holomorphic foliations having a suitable singularity in dimension two.

In this paper we prove versions of Poincar\'e-Lyapunov center theorem, including for the case of holomorphic vector fields. We also give some applications, hinting that there is much more to be explored in this framework.
\end{abstract}
\tableofcontents

\section{Introduction and main results}
\label{section:introduction}
An isolated singularity of a smooth vector field $X$ in dimension two is called a {\it center} if
it has a neighborhood consisting only of the singularity and periodic orbits (diffeomorphic to the
circle $S^1$). The classical Poincaré-Lyapunov center theorem reads as follows:

\begin{Theorem}[\cite{Lyapunov,Poincare}]
\label{Theorem:centertheorem} Consider a real analytic vector field $X(x,y)= P(x,y)\frac{\partial}{\partial x} + Q(x,y) \frac{\partial}{\partial y}$ defined in a  neighborhood of the  origin $O\in \mathbb{R}^2$, having a nonsingular linear part $DX(O)$ and a center at the origin. Then $X$ admits a real analytic first integral
 of Morse type in a neighborhood of  the origin.

\end{Theorem}
There are some equivalent statements  in terms of differential one-forms (\cite{leon-scardua}).
A quite geometrical proof was given by Moussu
(\cite{moussu}). In his paper he makes use of the complexification of the 1-form,
obtaining therefore a holomorphic 1-form with a suitable singularity at the origin $O\in \mathbb C^2$. To this complex 1-form it is applied  Mattei-Moussu theorem (\cite{mattei-moussu}) which  assures the existence of a holomorphic first integral near the singular point (\cite{mattei-moussu}, Theorem~B page 473).  Moussu's ideas are quite attractive and inspiring.
They also show the interplay between real analytic dynamical systems and the
geometric theory of holomorphic foliations.
In this paper we address problems motivated by the above statements, but relaxing the hypothesis on the existence of a neighborhood consisting only of period orbits.
In order to compensate this loss we require a resonance condition on the linear part of the vector field, as follows. We shall say that a 2x2 real matrix $A$ {\it generates a rotation} if all the orbits of the one-parameter group $exp(tA), t \in \mathbb R$, outside the origin of $\mathbb R^2$, are periodic and nontrivial. This corresponds to the fact that the eigenvalues of $A$ are purely imaginary of the form $\pm i \omega$ where $i ^2 =-1$. Then, up to multiplication of $X$ by a constant, we may assume that these eigenvalues are $\pm i$.

Our first result in this direction reads as follows:

\begin{theorem}
\label{Theorem:real}
Let $X$ be an analytic vector field in a neighborhood $U$ of the origin $O\in \mathbb R^2$. Assume that $X(O)=0$ and $DX(O)$ generates a rotation. Then  the following conditions are equivalent:

\begin{enumerate}[(i)]
\item The origin is a center singularity for $X$.
\item There is a sequence of periodic orbits $\gamma_\nu\subset U \setminus \{O\}, \nu \in \mathbb N$, of $X$ such that
     $\gamma_\nu \to O$ as $\nu \to \infty$ (in the sense of the Hausdorff topology).

\end{enumerate}

In case any of the conditions above is verified we conclude that $X$ admits an analytic first integral in the strong sense in a neighborhood of the origin.

\end{theorem}

Let us state a notion that shall be useful.

\begin{Definition}
{\rm Given a vector field $X$ in a neighborhood $U$  of the
origin $O\in \mathbb R^2$ we shall say that $X$ admits a transverse segment if there is a continuous injective map  $\phi \colon [0, \epsilon) \to \mathbb R^2$ such that: (i)  $\phi(0)=0$; (ii) $\phi\big|_{(0,\epsilon)}\colon (0,\epsilon) \to \mathbb R^2$ is a smooth immersion and (iii)
$\phi$ is transverse  to $X$ at any point off the origin.  In this case we shall simply say that $\Sigma=\phi(0,\epsilon)$ is a transverse segment to $X$. Given a point $p\in  \Sigma$ it has a {\it bounded order} by $k \in \mathbb N$ if the corresponding trajectory $\gamma_p$ of $X$ satisfies
$\sharp (\gamma_p \cap \Sigma)\leq k$.
}
\end{Definition}
This is clearly the case when we have a center type singularity. There are other examples: take a cusp singularity $x^2 - y ^3 =cte$ and the hamiltonian $X=3 y^2 \frac{\partial}{\partial x} + 2x \frac{\partial}{\partial y}$ and consider the vertical axes.

\vglue.1in
In the course of the proof of Theorem~\ref{Theorem:real} we shall obtain (cf. Lemma~\ref{Proposition:orbitsperiodic}):
\begin{theorem}
\label{Theorem:boundedorder}
Let $X$ be an analytic vector field in a neighborhood $U$ of the origin $O\in \mathbb R^2$. Assume that $X(0)=0$ and $DX(O)$ generates a rotation. Assume  that there is a sequence of  orbits $\gamma_\nu\subset U\setminus \{O\}$ of $X$   such that:
\begin{enumerate}
\item $\gamma_\nu$ has bounded order by $k\in \mathbb N$ with respect to some transverse segment $\Sigma$.
\item  $\gamma_\nu \to O$ as $\nu \to \infty$.
\end{enumerate}
Then  $X$ has a center type singularity at the origin. In particular,  $X$ admits a real analytic first integral in a neighborhood of the origin.
\end{theorem}

\subsection{Complex analytic foliations and real singularities}
In what follows, by a {\it germ of a holomorphic foliation at the origin $O\in \mathbb C^2$} we shall mean
a germ of a holomorphic foliation by curves, with an isolated singularity at the origin $0 \in \mathbb C^2$.
Two irreducible and reduced germs $f, g \in \mathcal O_2$ with $f(0)=g(0)=0$ are {\it in general position} if the analytic curves $(f=0)$ and $(g=0)$ meet transversely at the origin.
With this notion we have the following variant of Theorem~\ref{Theorem:real}:

\begin{theorem}
\label{Theorem:complex}
Let $\fa$ be a germ of a {\rm(}Siegel resonant type{\rm)} holomorphic foliation at the origin $0 \in \mathbb C^2$
 given by $\omega=0$ where  $\omega=xdy + ydx  + {\tilde \omega}$, where ${\tilde \omega}$ has jet of order one equal to zero.
Then the following conditions are equivalent:

\begin{enumerate}[{(i)}]

\item $\fa$ admits a holomorphic first integral of the form $fg$ for irreducible germs $f, g \in \mathcal O_2$ in general position.

\item There is a germ of an analytic dimension two variety $V^2\subset \mathbb C^2$ with $0 \in V^2$, having contact order one with  $\fa$ outside of the origin and
such that the restriction of $\fa$ to $V^2$ admits a sequence of compact leaves $L_\nu\subset V^2\setminus \{O\}$ such that $L_\nu \to O$ as $\nu \to \infty$.

\end{enumerate}

\end{theorem}

In the situation of the above theorem we also have:

\begin{itemize}

\item {\it There is   a germ of a totally real analytic variety $V^2\subset \mathbb C^2$ having contact order one with $\fa$ and
such that the restriction of $\fa$ to $V^2$  has a center type singularity at the origin  in $V^2$.}

\end{itemize}

A few words about the notions in the statement of Theorem~\ref{Theorem:complex}.
We recall that a  submanifold $V$ of a complex surface $M$ is called
{\it  totally real}  if the complex structure $J\colon TM \to TM$ of $M$ maps
each tangent space $T_p V\subset T_pM $ of $V$  into  the normal space $(T_p V)^\perp\subset T_p M$.
We refer to \cite{BER} for a detailed exposition about  totally real manifolds.
We mention that   given two germs of holomorphic functions $f, g \colon \mathbb C^2 \to \mathbb C$ in general position
and vanishing at $O\in \mathbb C^2$ then the intersection
$V^2 = (\Re(f)= \Re(g))\cap (\mI(f)=- \mI(g))$ is a germ of a  totally real surface at the origin $0 \in \mathbb C^2$.

In Theorem~\ref{Theorem:complex} above the leaves of $\fa$ are of real dimension two, in a space of real dimension four. Thus, condition (ii) is equivalent to the following:

\begin{itemize}
\item[{\rm (ii)'}] There is   a germ of a totally real analytic surface $V^2\subset \mathbb C^2$ with $0 \in V^2$ and
such that the restriction of $\fa$ admits a sequence of compact leaves $L_\nu\subset V^2\setminus \{O\}$ such that $L_\nu \to O$ as $\nu \to \infty$.
\end{itemize}

Given a real foliation $\fa$ of codimension $k$ in a differentiable manifold $M$ and an immersed connected submanifold $V\subset M$, the {\it contact order} of $\fa$ with $V$ at a point $p\in V$ is the dimension of the intersection $T_p(V)\cap T_p(\fa) \subset T_p(M)$ as linear subspaces of the tangent space $T_p(M)$. We say that $\fa$ has contact order $r$ with $V$ if their contact order is $r$ at each point $p\in V$. In the case where $\fa$ is a holomorphic foliation of (complex) codimension one in an open subset $U\subset \mathbb C^2$ with $\sing(\fa)=\{O\}\subset U$, and $V^2\subset U$ is a
real surface, we have:
\begin{itemize}
\item $V^2$ is transverse to $\fa$ off the origin iff $V^2\setminus\{O\}$ and $\fa$ have contact order equal to zero.

\item $V^2$ is $\fa$ invariant iff $V^2\setminus \{O\}$ and $\fa$ have contact order equal to $2$.

\item $V^2\setminus \{O\}$ has contact order with $\fa$ equal to $1$ iff $V^2\setminus \{O\}$ is a totally real submanifold not invariant by $\fa$.

\end{itemize}

For our next result it is better to state a definition:

\begin{Definition}
{\rm Let $\fa$ be a one-dimension holomorphic foliation with singularities on a Stein surface $N^2$. An isolated  singularity $p\in\sing(\fa)\subset  N$ is said to be  {\it real} if there is a germ at $p$ of an analytic dimension two variety $V^2\subset  N^2$, having contact order one with  $\fa$ outside of $p\in V$ and such that for some neighborhood $U\subset N^2$ of $p$, the restriction  $\fa\big|_{U}$ is the corresponding complexification of the foliation $\fa$ restricted to $V^2$.
We shall refer to the pair $(\fa\big|_{V}, V)$ (or sometimes just to the foliation $\fa\big|_{V}$) as {\it the real model} of (the germ of) $\fa$ at $p$.}
\end{Definition}

\begin{theorem}
\label{Theorem:complete}
Let $Z$ be a holomorphic vector field on a Stein surface $N^2$. Assume that:
\begin{enumerate}[(a)]
\item $Z$ is a complete holomorphic vector field.
\item There is a real singularity $p\in \sing(Z)$ such that the corresponding real model
is a nondegenerate   analytic center.

\end{enumerate}
Then $Z$ has a periodic flow, and admits a holomorphic first integral $f\colon N \to {\ov {\mathbb C}}$.
\end{theorem}

\begin{Remark}
{\rm By a {\it nondegenerate} center we mean a center singularity with nonsingular linear part.
\begin{enumerate}[(i)]
\item The conclusion of Theorem~\ref{Theorem:complete} remains valid if instead of assuming
(a) we assume that: $(a)^\prime$ $Z$ is of the form $Z=x\frac{\partial}{\partial y} - y \frac{\partial}{\partial x} + \hot$ in a neighborhood of $p$ and  the real model admits a sequence of compact leaves $L_\nu\subset V^2\setminus \{O\}$ such that $L_\nu \to p$ as $\nu \to \infty$.

\item The case of a complete holomorphic vector field $Z$ in a Stein surface, having a holomorphic first integral and a linearizable singularity of type $x y =constant$ has been described in \cite{camacho-scarduamanuscripta} in full details.

\end{enumerate}
}
\end{Remark}

\section{Basic lemmas}

Let us first state a few lemmas we shall need.
First we recall that given a topological space $X$, a point $p \in X$ and
  $h\colon U \to h(U)\subset X$   a homeomorphism between $U$ and $h(U)$ open subsets of $X$, such that $h(p)=p$,   we can define the {\it pseudo-orbit} of a point $q \in U$ as  the set of all possible iterates $h^n(q)\in U, \, n \in \mathbb Z$.
We shall say that the pseudo-orbit of $q\in U$ is  {\it closed in U} if its a closed subset of $U$ in the classical sense of topology. This means   either of the following. There are only finitely many possible iterates of $q$ or if any point $z \in U$ which is a limit of a sequence of iterates $z=\lim h^{k_j}(q)$ of some point $q\in U$, with $k_j \in \mathbb N$ and $\lim k_j = \infty$ then $z$ belongs to the pseudo-orbit. We shall say that the {\it the orbit of $p$ is periodic} (or, that {\it $p$ is a periodic point}) of period $k\ge 1$, with respect to $h$,  if $f^\ell (p) \in U, \forall \ell \in \mathbb N$,   $f^k(p)=p$ and $f^\ell(p) \ne p, \forall \ell=0,...,p-1$.
Using representatives we shall state  similar notions for germs of homeomorphisms with a fixed point. For the case of a complex diffeomorphism map germ we have:
\begin{Lemma}
\label{Lemma:diffeomorphism}
Let $f\in \Diff(\mathbb C,0)$ be a germ of a holomorphic diffeomorphism
such that:
\begin{enumerate}
\item $f^\prime(0)\in \mathbb C$ is a root of the unit.
\item There is a sequence of  points $p_\nu\ne 0$ such that
$p_\nu$ is periodic with respect to $f$ and $p_\nu \to 0 $ as $\nu \to \infty$.
 \end{enumerate}
 Then $f$ has finite order, i.e., $f^k=Id$ for some $k \in \mathbb N \setminus \{0\}$.

\end{Lemma}

\begin{proof}
Since $f$ has a periodic linear part, there is a smaller positive integer $k \in \mathbb N\setminus \{0\}$ such that $f^k (z) = z + (...)$. If $f$ does not have finite order then we must have
$f^k(z) = z + a_{l+1} z^{l+1} + a_{l+2} z^{l+2}+...$ for some $a_{l+1}\ne 0$ and $l >0$. By replacing $f$ by  $f^k$ we may then assume that $f$ is tangent to the identity $f^\prime(0)=1$ but
$f\ne Id$. Using now Camacho theorem (\cite{camacho}) we conclude that no point $p\ne 0$ has a
periodic orbit, contradiction.
\end{proof}

An  extension of the above lemma, with a similar proof  is:

\begin{Lemma}
\label{Lemma:diffeomorphismclosed}
Let $f\in \Diff(\mathbb C,0)$ be a germ of a holomorphic diffeomorphism
such that:
\begin{enumerate}
\item $f^\prime(0)\in \mathbb C$ is a root of the unit.
\item There is a sequence of  points $p_\nu\ne 0$ such that the pseudo-orbit of
$p_\nu$ is finite with uniformly bounded order, with respect to $f$, and $p_\nu \to 0 $ as $\nu \to \infty$.
 \end{enumerate}
 Then $f$ has finite order, i.e., $f^k=Id$ for some $k \in \mathbb N \setminus \{0\}$.

\end{Lemma}

\begin{proof}
As in the preceding proof we may assume that $f(z)= z + a_{k+1}z^{l+1} + \hot ,\, \,  a_{l+1} \ne 0$ and look for a contradiction. The topological description of the orbits of $f$ (\cite{camacho}) then shows that the closer to the origin, the more different iterates a point will have and, given any $k \in \mathbb N$, and no matter which neighborhood $0\in V \subset \mathbb C$ we choose, there is a small closed disk $0 \in D \Subset V$ centered at the origin, such that for every point $p \in D \setminus \{O\}$ the orbit of $p$ has order at least $k^2$.  This gives the desired contradiction.

\end{proof}

\section{Complexification of foliations, proof of  Theorems~\ref{Theorem:real} and ~\ref{Theorem:complex}}

Let $\fa$ be a real analytic codimension one foliation with singularities in a neighborhood of the origin $0 \in \mathbb R^n$. This means that $\fa$ is defined by a real analytic 1-form
$\omega=\sum\limits_{j=1}^n a_j(x)dx_j$, defined in a neighborhood of the origin, and satisfying the integrability condition $\omega \wedge d\omega=0$. We consider the complexification of $\fa$ which we denote by $\fa_{\mathbb C}$. This is a codimension one holomorphic foliation with singularity, defined in a neighborhood of the origin $O\in \mathbb C^n$ by the complexification $\omega_{\mathbb C}$ of the form $\omega$.
In complex coordinates $(z_1,\ldots,z_n)$ we can write $z_j=x_j + i y_j$ and $\omega_{\mathbb C}=
d(\sum\limits_{j=1}^n z_j ^2) + {\tilde \omega}_{\mathbb C}$ for some 1-form ${\tilde \omega}_{\mathbb C}$ with zero first jet at the origin. Now we consider the real space $\mathbb R^n \subset \mathbb C^n$ given by $y_j=0, j=1,\ldots,n$.

The next result is a well-known easy to prove lemma:
\begin{Lemma}
\label{Lemma:complexificationfirstintegral}
Let $\fa$ be a real analytic foliation in a neighborhood of the origin $0 \in \mathbb R^n$ whose complexification $\fa_{\mathbb C}$ admits a
holomorphic first integral. Then $\fa$ admits a real analytic first integral, defined in some neighborhood of the origin.
Indeed, there is a real analytic first integral $f$ for $\fa$ such that the complexification $f_{\mathbb C}$ of $f$ is a
holomorphic first integral for $\fa_{\mathbb C}$.
\end{Lemma}

The main point is the following:

\begin{Proposition}
\label{Proposition:orbitsperiodic}
Let $X$ be a real analytic vector field in a neighborhood $U$ of the origin $O\in \mathbb R^2$, such that $X(0)=0$ and $DX(O)$ generates a rotation. Assume also that there exists a sequence of periodic  orbits $\gamma_\nu\subset U \setminus \{O\}$ of $X$ converging to $0$  (in the classical sense of Hausdorff topology). Then the origin is a center type singularity for $X$.

\end{Proposition}

\begin{proof}The complexification $X_{\mathbb C}$ of $X$ is a complex analytic vector field defined in a neighborhood of the origin $O\in \mathbb C^2$. As we have remarked before, we may assume that the linear part $DX_{\mathbb C}(0)$ has  eigenvalues given by $\pm i$.
 We may therefore find suitable complex coordinates $(x,y)\in V \subset \mathbb C^2$ such that
  $X_{\mathbb C}= x \frac{\partial}{\partial x} -y \frac{\partial}{\partial y} + X_2$ where $X_2$ has a zero order one jet at the origin.
Then $X_{\mathbb C}$ generates a holomorphic foliation $\fa_{\mathbb C}$ with an isolated Siegel type singularity at the origin, of the form $ydx + xdy + \ldots=0$. The germ of foliation $\fa_{\mathbb C}$ is in the Siegel domain and we may assume that the coordinate axes are invariant (\cite{mattei-moussu}). In this case the quadratic  blow-up $\pi \colon \tilde {\mathbb C^2} \to \mathbb C^2$ at the origin $O\in \mathbb C^2$,  induces a foliation
$\tilde {\fa}:=(\fa_{\mathbb C})^*$ in a neighborhood of the exceptional divisor $\pi^{-1}(0)= E$ in the blow-up space $\widetilde{\mathbb C^2 _0}$. The foliation $\tilde \fa$  leaves invariant the exceptional divisor $E\simeq \mathbb CP(1)$ and has exactly two singularities, the north and south poles, in $E$, both of Siegel resonant type (indeed, these singularities and the structure of the restriction of $\tilde \fa$ to $E$ are determined by the linear part $nydx + mxdy=0$ of $\fa$).
Let us study this a bit more thoroughly.
Given complex coordinates $(x,y)\in \mathbb C^2$, we consider the {\it real plane}
$\mathbb R^2 \subset \mathbb C^2$ as given by $\Im(x)=\Im(y)=0$. The inverse image of the real plane in the blow-up $\tilde{\mathbb C^2 _0}$ corresponds to a Moebius band ${\mathcal M}^2\subset \tilde{\mathbb C^2}$, intersecting the exceptional divisor transversely at the equator $\mathcal E$  of the exceptional divisor $\mathbb E\simeq \mathbb CP(1)$. The pull-back foliation $\fa^*$
in $\tilde{\mathbb C^2 _0}$  leaves invariant this Moebius band.
Now we consider the projective holonomy group of the exceptional divisor $E$. This means the {\it holonomy group} (\cite{C-LN})  of the leaf $E\setminus \sing(\fa^*)$ for the foliation $\fa^*$. From what we have seen above, this foliation has exactly two singularities in $E$, corresponding to the north and south poles of $E$. Thus the holonomy group above mentioned is generated by a simple loop around the equator, i.e, this is a cyclic group.
Let us denote by $h$ a generator of this group obtained as follows. Choose a point $p\in E$ and a local transverse disc $\Sigma$ to $\mathbb E$ centered at $p$. Then denote by $H\colon (\Sigma,p) \to (\Sigma,p)$ the holonomy map corresponding to
the equator $\gamma={\mathcal M}^2 \cap \mathbb E$. Notice that, since $\mathbb E$ is invariant by $\fa^*$, the equator $\mathcal E$ corresponds to a  compact leaf (periodic orbit) of the induced foliation in ${\mathcal M}^2$. A transverse open segment $\sigma \simeq (-1,1)\subset \mathcal M^2$, transverse to $\mathcal E$ at a base point $p\in \mathcal E$ will then induce  an associate first return map (Poincaré map) $h\colon (\sigma,p) \to (\sigma,p)$ corresponding to the periodic orbit
$\mathcal E$. Moreover, each closed (i.e., periodic) orbit $\gamma$ of $X$ in $\mathbb R^2$ lifts into a closed compact curve $\tilde \gamma $ in the Moebius band ${\mathcal M}^2$. This closed curve $\tilde \gamma$ then corresponds to a periodic orbit for the holonomy map $h$, this orbit consisting of only two elements, except for the case of $\gamma = \mathcal E$ which gives the fixed point of $h$.
Thus we have a germ of a real  analytic map   $h\in \Diff^w(\mathcal E,p)\simeq \Diff^w(\mathbb R,0)$. This map has a sequence of periodic points (of period 2) accumulating at the origin. By the identity principle, $h$ must be periodic of period $2$, i.e., $h^2 = Id$. This already  implies that the $\fa^*$-holonomy map $H$ admits a real analytic curve $\gamma \cap \Sigma$ where its orbits are periodic of period $\leq 2$. Since $\gamma$ contains the origin,  $H$ is a periodic map
of period $2$. This implies, by standard methods described in \cite{moussu} and from Mattei-Moussu's theorem (\cite{mattei-moussu} page 473) the foliation $\fa_\mathbb C$ admits a holomorphic first integral in a neighborhood of the origin $O\in \mathbb C^2$. From Lemma~\ref{Lemma:complexificationfirstintegral} we conclude that the vector field $X$ admits an analytic first integral. Let us denote by $f\colon U,0 \to \mathbb R,0$ an analytic first integral of $X$. This means that $X(f)=0$, i.e., $f$ is constant on each orbit of $X$ in $V$.
Thanks to the linear part of $X$ we may assume that $f(x_1,x_2) = x_1 ^2 + x_2 ^2 +$higher order terms and thanks to Morse lemma
we conclude that the origin is a center singularity for $X$.

\end{proof}

\begin{proof}[Proof of Theorem~\ref{Theorem:real}]
The main point is $(ii) \implies (i)$.
According to Lemma~\ref{Proposition:orbitsperiodic} the origin is a center singularity. Evoking then Lyapunov-Poincaré theorem (Theorem~\ref{Theorem:centertheorem}) we conclude that $\fa$ admits a real analytic first integral.
\end{proof}

\begin{proof}[Proof of Theorem~\ref{Theorem:boundedorder}]
We proceed similarly to Theorem~\ref{Theorem:real}. Indeed, Lemma~\ref{Lemma:diffeomorphism} is now replaced by Lemma~\ref{Lemma:diffeomorphismclosed} in order to, with the use of an easy adaptation of  Proposition~\ref{Proposition:orbitsperiodic},  conclude that the singularity is a center. The remaining part follows as usual.
\end{proof}

\section{Holomorphic foliations: proof of Theorems~\ref{Theorem:complex} and ~\ref{Theorem:complete}}
\label{section:holomorphic}

Let us now prove Theorem~\ref{Theorem:complex}.

\begin{proof}[Proof of Theorem~\ref{Theorem:complex}]
We proceed as in \cite{leon-scardua}. Assume that $\fa$ admits a holomorphic first integral of the form $fg$ with $f, g \in \mathcal O_2$, $f(0)=g(0)=0$,
$f$ and $g$ (being germs reduced and  irreducible and)
in general position. We consider the analytic varieties of real codimension one $\mathcal  R: (\Re f = \Re g)\subset \mathbb R^4$ and $\mathcal I : (\mI f= - \mI g)\subset \mathbb R^4$. Since $f$ and $g$ are in general position the intersection
$\mathcal R \cap \mathcal I =V^2$ is a two-dimensional analytic variety. Also $0 \in V^2$ because $f$ and $g$ vanish at the origin. Let us now put $X=\frac{f+ g}{2}$ and $Y=\frac{f - g }{2 i}$. Then $f= X + iY$ and $g = X - i Y$ and therefore
$fg=X^2 + Y^2$. Moreover, in the variety $V^2$ we have $X=\Re(f)=\Re(g)$ and $Y=\mI(f)=-\mI(g)$ so that, restricted to $V^2$ we have
$fg=||f||^2=||g||^2$. This shows that the restriction to $V^2$ of the foliation  $\fa$ is a real analytic foliation by curves which are closed. In particular, the contact order of $\fa$ with $V^2$ is one. Indeed the restriction $\fa\big|_{V^2}$  gives an analytic center type singularity at the origin $0\in V^2$. Finally, since $\fa$ is holomorphic and has contact order equal to one with $V^2$ it follows that  $V^2$ is a totally real subvariety.
This proves $(i) \implies (ii)$ in Theorem~\ref{Theorem:complex}.

Let us now prove $(ii) \implies (i)$. From hypothesis (ii) and from the considerations after Theorem~\ref{Theorem:complex}:
{\em there is   a germ of a totally real analytic variety $V^2\subset \mathbb C^2$ having contact order one with $\fa$ and such that the restriction of $\fa$ to $V^2$  has a nondegenerate linear part  singularity at the origin  in $V^2$ which is the limit $p=\lim_{\nu \to \infty} L_\nu$  of a sequence of compact leaves $L_\nu$ of the restriction $\fa\big|_V$.}
Up to an analytic change of coordinates in $\mathbb C^2$ we may assume that $V^2\subset \mathbb C^2$ corresponds to the totally real space $\mathbb R^2 \subset \mathbb C^2$, i.e., in suitable local coordinates $(x,y)\in \mathbb C^2$ we have $V^2: (\mI (x) = \mI (y)=0)$.
By hypothesis the holomorphic foliation $\fa$ is defined in a neighborhood of the origin $O\in \mathbb C^2$ by a 1-form $\omega=d(xy) + {\tilde \omega}$ where ${\tilde \omega}$ has zero jet of order one at the origin.

Now we claim:
\begin{Claim}
In these coordinates $(x,y)$, $\fa$ is the complexification of a real analytic foliation $\fa_{\mathbb R}$ which has a center type singularity at the origin $O\in \mathbb R^2$.
\end{Claim}
\begin{proof}[Claim]
 $\fa$ has contact order one with  the real space $\mathbb R^2 \subset \mathbb C^2$ and its restriction to this space exhibits a nondegenerate linear part  singularity at the origin $O\in \mathbb R^2$. Recall that the real space above is given by $\mI(x)=\mI(y)=0$ where $(x,y)\in \mathbb C^2$ are affine
coordinates in $\mathbb C^2$. Moreover, still by hypothesis, the foliation  $\fa^*$ in ${\mathcal M}^2$ admits a sequence of leaves $L_\nu, \nu \in \mathbb N$ such that each $L_\nu$ is  compact and $L_\nu \to \gamma$ as $\nu \to \infty$. By Proposition~\ref{Proposition:orbitsperiodic} this real foliation has a center singularity at the origin, this singularity admitting an analytic first integral. Also the complex foliation $\fa$ admits a holomorphic first integral.
\end{proof}
As we have seen above, $\fa$ admits a holomorphic first integral.
 It is not difficult to use the linear part of $\fa$ to conclude that  this first integral is of the form $fg$ where $f, g \in \mathcal O_2$ are irreducible and reduced and, up to reordering $f$ and $g$, we must have  $x\big|f$ and $y\big| g$ in $\mathcal O_2$.
 \end{proof}

\begin{proof}[Proof of Theorem~\ref{Theorem:complete}]
First we recall that the main reference for holomorphic flows on Stein spaces is the work of
M. Suzuki (\cite{Suzuki1, Suzuki2}). In particular, from Suzuki's work we know that there is a {\it typical orbit} for the vector field in the following sense: there is a zero logarithmic capacity subset $\sigma \subset N$ such that $\sigma$ is invariant, and every orbit of $Z$ off $\sigma$ is diffeomorphic to $R$ where $R$ is a Riemann surface belonging to the following list $\{\mathbb C, \mathbb C^*\}$.
Moreover, still according to Suzuki, in case the typical orbit is diffeomorphic to $\mathbb C^*$ the foliation $\fa(Z)$ admits a meromorphic first integral $f\colon N^2 \to  {\ov{\mathbb C}}$.

Now we observe that, from our hypotheses and from the proof of Theorem~\ref{Theorem:complex}, the real model induces a complexification that admits a holomorphic first integral in  a neighborhood of $p$ which is a singularity of the form $xdy + ydx + \hot =0$ in suitable holomorphic coordinates
$(x,y)\in U \subset N^2$, centered at $p$. This singularity admits a holomorphic first integral and it is therefore analytically linearizable. This implies by \cite{camacho-scarduamanuscripta} that the flow is periodic, having therefore typical orbit diffeomorphic to $\mathbb C^*$ and, by Suzuki, it admits a meromorphic first integral $f \colon N^2 \to {\ov{\mathbb C}}$.

\end{proof}

\bibliographystyle{amsalpha}

\begin{thebibliography}{31}
\frenchspacing

\bibitem{BER} M.S. Baouendi, P. Ebenfelt, L.P. Rothschild: {\it Real submanifolds in complex space and their mappings}. Princeton Mathematical Series, 47. Princeton University Press, Princeton, NJ, 1999.
\bibitem{camacho} C.  Camacho: {\it On the local
structure of conformal maps and holomorphic vector fields in
 ${\mathbb C}^2$}, Joun\'ees Singuli\'eres de Dijon (Univ. Dijon,
Dijon, 1978, 3, 83--94), Ast\'erisque 59-60, Soc.  Math.  France,
Paris 1978.


\bibitem{camacho-scarduamanuscripta} C. Camacho \& B. Scárdua, Nondicritical $\mathbb C^*$-actions on
two-dimensional Stein manifolds; Manuscripta mathematica volume 129, pages 91-98 (2009)


%\bibitem{C-LN-S1} C. Camacho, A. Lins Neto and P. Sad: {\it   Topological invariants and
%equidesingularization for holomorphic vector fields}; J. of Diff. Geometry, vol.
%20, no.\ 1 (1984), 143--174.


%\bibitem{Cerveau-LinsNeto} D. Cerveau and A. Lins Neto:
%{\it Local Levi-flat hypersurfaces invariants by a codimension
%one holomorphic foliation}.
%Amer. J. Math. 133, (2011), no. 3, 677-716


%\bibitem{Cerveau-Loray} D. Cerveau and F. Loray: {\it Un théorème de
%Frobenius singulier via l´arithmétique
%élémentaire}. J. Number Theory, 68(2) (1998), 217-228.


%\bibitem{Cerveau-Scardua2} D. Cerveau and B. Scardua: {\it Integrable deformations of foliations:
 %a generalization of Ilyashenko's result}, pre-print Rennes 2018.
 %To appear in Moscow Mathematical Jornal $\geq 2020$.

\bibitem{C-LN} C. Camacho and A. Lins Neto: {\it Geometric theory of foliations}, Translated from the Portuguese by Sue E. Goodman. Birkhauser Boston, Inc., Boston, MA (1985).


%\bibitem{Chern}  B. Y. Chen and K. Ogiue: {\it On totally real submanifolds}.
%Trans. Amer. Math. Soc.
%193 (1974) 257-266.

%\bibitem{deligne} P. Deligne: {\it Le groupe
%fundamental du compl\'ement
%d'une courbe plane n'ayant que des points doubles ordinaires est
%ab\'elien}. S\'eminaire Bourbaki, Vol. 79/90 nov. 1979.


%\bibitem{gunning1} R. C. Gunning: { Introduction to
%holomorphic functions of several variables}, vol. I, Function Theory.
%Wadsworth \& Brooks/Cole Advanced Books \& Software, Pacific Grove, CA, 1990.



%\bibitem{gunning-rossi} R. C. Gunning and H. Rossi: Analytic functions of several complex
%variables, Prentice Hall, Englewood Cliffs, NJ, 1965.


%\bibitem{Le-Saito} L. D. Tráng and K. Saïto: {\it The local $\pi_1$ of the
%complement of a hypersurface
%with normal crossings in codimension $1$ is abelian}, Ark. Math. 22
%(1984), no. 1, 1-24. MR 735874 (36a:32019)

%\bibitem{Leau} L. Leau: {\it Étude sur les equations fonctionnelles à une ou à plusieurs variables}.
%Ann. Fac. Sci. Toulouse Math. Sér. I. 11 (1897), 1-110.


\bibitem{leon-scardua} V. León, B. Scárdua, {\it On a Theorem of Lyapunov-Poincaré in Higher Dimensions},
July 2021, Arnold Mathematical Journal 7(3)
DOI:10.1007/s40598-021-00183-x.

%\bibitem{malgrangeII} B. Malgrange: {\it  Frobenius avec singularites. 2. Le cas general}.
%Inventiones mathematicae (1977)
%Volume: 39, page 67-90.



\bibitem{mattei-moussu} J.F. Mattei and R. Moussu: {\it  Holonomie et int\'egrales premi\`eres}, Ann. Sci. \'Ecole Norm. Sup. (4) {\bf 13} (1980), 469--523.

%\bibitem{Milnor} J. Milnor: {\it Singular points of complex hypersurfaces}. Ann.
%Math. Studies, Vol. 61, Princeton Univ. Press, 1968.


\bibitem{moussu} R. Moussu: {\it Une démonstration géométrique d'un théorème
de Lyapunov-Poincaré}. Astérisque, tome 98-99 (1982), p. 216-223.

\bibitem{Reeb} G. Reeb: {\it Sur certaines propri\'et\'es
topologiques des vari\'et\'es feuillet\'ees};
Actualit\'es Sci. Ind., Hermann, Paris, 1952.


%\bibitem{saito} K. Saïto: {\it On a generalization of de Rham lemma.}
%Annales de l'institut Fourier (1976), Volume: 26, Issue: 2, page
%165-170,


\bibitem{Lyapunov} A. Lyapunov: {\it Etude d'un cas particulier du problème de la stabilité du
    mouvement}. Mat. Sbornik 17 (1893) pages 252-333 (Russe).

\bibitem{Poincare} H. Poincaré: {\it Mémoire sur les courbes définies par une équation différentielle (I)},
Journal de mathématiques pures et appliquées 3e
série, tome 7 (1881), p. 375-422.



\bibitem{Suzuki1} M. Suzuki: {\em Sur les op\'erations holomorphes de
$\bc$ et de $\bc^*$ sur un espace de Stein},
 S\'eminaire Norguet, Springer Lect. Notes, 670 (1977), 80-88.


\bibitem{Suzuki2} M. Suzuki:
{\em Sur les op\'erations holomorphes du groupe additif complexe
sur l'espace de deux variables complexes}; Ann. Sci. \'Ec. Norm.
Sup. 4 $^e$ s\'erie, t.10, 1977, p. 517 \`a 546.


\end{thebibliography}

\end{document}